\documentclass[10pt]{amsart}

\title[Integral Hodge atoms]{Integral Hodge structures \\ and atoms of fourfolds}
\author[Gu\'er\'e]{J\'er\'emy Gu\'er\'e}
\address{Univ. Grenoble Alpes, CNRS, IF, 38000 Grenoble, France}
\email{jeremy.guere@univ-grenoble-alpes.fr}

\usepackage{amsrefs}
\usepackage{amsfonts}
\usepackage{amsmath}
\usepackage{amssymb}
\usepackage{amscd}
\usepackage{array}
\usepackage{subfigure}
\usepackage{tikz}
\usetikzlibrary{arrows}
\usepackage{pifont}
\usepackage{wasysym}

\usepackage[all]{xy}

\allowdisplaybreaks[3]

\newcommand{\cN}{\mathcal{N}}

\newcommand{\un}{\textbf{1}}

\newcommand{\ci}{\mathrm{i}}
\newcommand{\PP}{\mathbb{P}}
\newcommand{\CC}{\mathbb{C}}
\newcommand{\ZZ}{\mathbb{Z}}
\newcommand{\NN}{\mathbb{N}}

\newcommand{\QQ}{\mathbb{Q}}

\newcommand{\ev}{\mathrm{ev}}

\renewcommand{\(}{\left(}
\renewcommand{\)}{\right)}

\newcommand{\cF}{\mathcal{F}}

\theoremstyle{plain}

\newtheorem{thm}{Theorem}
\newtheorem*{thm*}{Theorem}
\newtheorem{pro}[thm]{Proposition}
\newtheorem{lem}[thm]{Lemma}
\newtheorem*{lem*}{Lemma}

\newtheorem{cor}[thm]{Corollary}

\newtheorem*{conv*}{Convention}

\theoremstyle{definition}

\newtheorem{dfn}[thm]{Definition}
\newtheorem{nt}[thm]{Notation}
\newtheorem*{nt*}{Notation}

\newtheorem{rem}[thm]{Remark}
\newtheorem*{rem*}{Remark}

\newtheorem*{rems*}{Remarks}

\newtheorem*{exa*}{Examples}
\newtheorem*{exait*}{\rm \em Example}
\newtheorem*{exadefit*}{\rm \em Example/Definition}
\newtheorem*{cla*}{\rm \em Claim}

\newtheorem*{dfn*}{Definition}

\usepackage{amsxtra}

\newcommand{\even}{\mathrm{even}}

\newcommand{\Eu}{\mathrm{Eu}}
\newcommand{\Hdg}{\mathrm{Hdg}}

\def\<{\left\langle}
\def\>{\right\rangle}

\newcommand{\id}{\mathrm{id}}

\begin{document}
\begin{abstract}
Following the recent work by OpenAI on the irrationality of very general cubic fourfolds in Hassett divisors, we prove some compatibility of Iritani's blow-up formula with the classical blow-up decomposition of integral Hodge structures.
As a consequence, we express OpenAI's birational invariant in terms of the theory of atoms developped by Katzarkov--Kontsevich--Pantev--Yu.
We also show that, for every smooth complex rational cubic fourfold, the primitive cohomology is isometric as an integral Hodge structure to the twisted middle cohomology of a K3 surface and we prove irrationality of the very general Verra fourfold.
\end{abstract}

\maketitle

\tableofcontents

\setcounter{section}{-1}

\section{Introduction}
In their seminal work \cite{KKPY}, Katzarkov--Kontsevich--Pantev--Yu developped the theory of atoms as a new tool in birational geometry, especially in dimension up to four.
As an illustration of their powerful theory, they prove irrationality of the very general cubic complex fourfold.

The case of smooth cubics fourfolds is indeed especially famous.
Hassett \cite{Hassett} described specific divisors in their moduli space for which the primitive cohomology of the cubic is isometric as a $\ZZ$-Hodge structure to the twisted middle cohomology of a K3 surface.
Conjecturally, it was a necessary and sufficient condition for the cubic to be rational.
An equivalent conjecture due to Kuznetsov \cite{Kuznetsov} is that a smooth cubic fourfold is rational if and only if its derived category contains a component in its orthogonal decomposition, the so-called \textit{Kuznetsov component}, that is isomorphic to the derived category of a K3 surface.
As first evidences towards these conjectures are the works on the very general cubic \cite{KKPY}, on the rational Hodge structures of rational cubics \cite{JG}, and on irrationality of very general cubics in some divisors where the $\QQ$-Hodge structure is of K3-type, but the $\ZZ$-Hodge structure is not \cite{Fay2}.

Very recently, OpenAI has released a preprint \cite{OpenAI} claiming that these conjectures are false.
Specifically, the very general member in Hassett divisors of large discriminant should be irrational.
This preprint has not yet been implemented in Lean and we were not able to follow all the arguments in the proof.
It is divided into two parts:
\begin{itemize}
\item the first part relies on genus-zero Gromov--Witten (GW) theory,
\item the second part relies on the Sarkisov program using the work of \cite{Lin}.
\end{itemize}
We mainly focused our reading on the first part, for which the most crucial result is \cite[Theorem 3.3]{OpenAI}. It takes the form of a new birational invariant mixing genus-zero GW theory and integral Hodge structures without refering explicitly to the theory of atoms.
We prove here that it is actually a biproduct of the theory of atoms \cite{KKPY}.

Our goal is then to re-write a proof of \cite[Theorem 3.3]{OpenAI} and to show that it can be deduced from the blow-up formula in quantum cohomology proven by Iritani \cite{Iritani}, and therefore that the birational invariant in \cite{OpenAI} is already part of the theory of atoms.
It is worth noticing that the original proof of \cite[Theorem 3.3]{OpenAI} follows the same path as Iritani's proof.
Specifically, both aim at controling the behaviour of GW invariants under a blow-up map, i.e.~to relate invariants of $X:=\mathrm{Bl}_Z(Y)$ to those of $Y$ and of $Z$, where $Z \subset Y$ is a smooth subvariety of a smooth projective complex variety. They both work on the deformation to the normal cone $\mathrm{Bl}_{Z \times 0}(Y \times \PP^1)$, which is the relevant master space in this context.
Then Iritani uses Givental's formalism to work out recursive formulae relating GW invariants of $X$, $Y$, and $Z$, whereas OpenAI uses the degeneration formula.
Both are very classical methods in genus-zero GW theory.
A third natural method would be using the virtual localization formula.
In our opinion, Givental's formalism encodes the recursion in a better way (and Iritani's formula is indeed a richer and more precise statement than the result obtained by OpenAI), but OpenAI could use its extraordinary computational capability to run the induction as well, although it looks quite painful to read; it is a recursion on five distinct indices.
We suspect that a proof via the virtual localization formula would involve similar complexity.

After trying to read OpenAI's proof, we decided to shift our perspective and prove its result directly by using Iritani's work.
In the end, a small miracle appears in dimension four: Iritani's formula gets simplified and becomes compatible with the classical decomposition of integral Hodge structures under a blow-up map, see Proposition \ref{initial}.

We thus obtain that our result in \cite{JG} can be upgraded to an isometry of $\ZZ$-Hodge structures, namely:
\begin{thm*}[see Theorem \ref{cubic}]
Let $X$ be a rational smooth complex cubic fourfold.
Then there exists a projective K3 surface $\Sigma$ and an isometry of integral Hodge structures
$$H^4(X,\ZZ)_\mathrm{prim} \simeq H^2(\Sigma,\ZZ)(-1).$$
\end{thm*}

Similarly, we prove the same result for Verra fourfolds, yielding the following improvement of \cite{Fay}:
\begin{thm*}[see Corollary \ref{Verra}]
A very general smooth complex Verra fourfold is irrational.
\end{thm*}

\section{Iritani's blow-up formula}
In the whole paper, all varieties are assumed to be smooth and projective over the field of complex numbers. Unless otherwise stated, $K$ refers to a number field.

\subsection{Quantum cohomology}
Let $X$ be a variety.
Genus-zero Gromov--Witten invariants of $X$ produce the quantum multiplication in cohomology and the Euler vector field.
Precisely, given a (homogeneous) basis $\(\alpha_0,\dotsc,\alpha_h\)$ of the $\QQ$-Hodge locus $H^*(X,\QQ)_\Hdg$ and the associated formal variables $T_0, \dotsc, T_h$, we define the parameter
$$\tau = T_0 \alpha_0 + \dotsb + T_h \alpha_h$$
and the Euler field is
$$\Eu_\tau = c_1(X) + \sum_{k=0}^h \(1-\frac{\deg(\alpha_k)}{2}\) T_k \alpha_k.$$
Given a (homogeneous) basis $(e_0,\dotsc,e_m)$ of the cohomology $H^*(X,\QQ)$, the quantum product is
$$e_i \star_\tau e_j = \sum_{k,n,\beta} \langle e_i, e_j, e^k, \tau^{\otimes n} \rangle_{0,3+n,\beta} \frac{Q^\beta}{n!} e_k,$$
where $Q$ is the Novikov variable and $(e^0,\dotsc, e^m)$ is the dual basis with respect to the Poincar\'e pairing.
At last, we define the endomorphism
$$\kappa^X_\tau(\cdot) = \Eu_\tau \star_\tau (\cdot) \in H^*(X,\QQ)_{R^*(X,\QQ)},$$
where $R^*(X,\QQ)$ is the graded $\QQ$-algebra $\QQ[[Q,T_0,\dotsc,T_h]]$.
Since GW invariants are defined using an algebraic cycle, then it is a morphism of $\QQ$-Hodge structures.

\begin{dfn}
We call \textit{transcendantal lattice} the orthogonal complement $T_X$ of the $\ZZ$-Hodge locus in $X$.
Precisely, it is
$$T_X := \(H^*(X,\ZZ)/\mathrm{tor}\) \cap \(H^*(X,\QQ)_\Hdg\)^\perp,$$
where we mod out the integral cohomology by its torsion part and we take the orthogonal with respect to the Poincar\'e pairing.
We simply denote by $T_{X,\QQ}$ the orthogonal complement of the $\QQ$-Hodge locus.
\end{dfn}

We have an orthogonal decomposition
$$H^*(X,\QQ) = H^*(X,\QQ)_\Hdg \oplus T_{X,\QQ}$$
and we denote by $\mathrm{pr}_{T_{X,\QQ}} \colon H^*(X,\QQ) \to T_{X,\QQ}$ the projection.
Then the endomorphism
$$\kappa^{X,\mathrm{trans}}_\tau \colon T_{X,\QQ} \to T_{X,\QQ}$$
is the composition $\mathrm{pr}_{T_{X,\QQ}} \circ \(\kappa^X_\tau\)_{|T_{X,\QQ}}$.
It is also a morphism of $\QQ$-Hodge structures.

\begin{dfn}
The $\QQ$-algebra $G_{X,\QQ}$ of the variety $X$ is the $\QQ$-algebra of the coefficients of $\kappa^{X,\mathrm{trans}}_\tau$ with respect to the formal variables $Q$, $T_0, \dotsc, T_h$, i.e.~
$$G_{X,\QQ} = \QQ\left[ \mathrm{Coeff}_{T_0^{a_0} \dotsm T_h^{a_h}Q^\beta}\(\kappa^{X,\mathrm{trans}}_\tau\)\right]_{a_0, \dotsc, a_h \in \NN, \beta \in \mathrm{NE}_\NN(X)} \subset \mathrm{End}_{\QQ-\mathrm{HS}}(T_{X,\QQ}).$$
\end{dfn}

\begin{nt}
For any $k \in \ZZ$, the subspace of cohomology classes of Hochschild degree-$k$ is
$$H^{(k)}(X) = \bigoplus_{p-q=k} H^{p,q}(X).$$
The subspace of cohomology classes of even (resp.~odd) cohomology degrees is denoted $H^\even(X)$ (resp.~$H^\mathrm{odd}(X)$).
We use the notations
\begin{eqnarray*}
T_{X,\QQ}^+ & = & T_{X,\QQ} \cap H^\even(X,\QQ), \\
T_{X,\QQ}^- & = & T_{X,\QQ} \cap H^\mathrm{odd}(X,\QQ).
\end{eqnarray*}
\end{nt}

\begin{rem}
The endomorphism $\kappa_\tau$ respects the decomposition of cohomology classes into even/odd degrees and its complexification respects the decomposition into Hochschild degrees, because it is based on integration over an algebraic cycle, the so-called \textit{virtual fundamental class}.
Therefore, the same property holds for the endomorphism $\kappa^\mathrm{transc}_\tau$.
Accordingly, we have a decomposition
$$G_{X,\QQ} = G_{X,\QQ}^+ \oplus G_{X,\QQ}^-$$
into even and odd parts, and the complexification $G_{X,\CC}$ of the $\QQ$-algebra $G_{X,\QQ}$ further decomposes into
$$G_{X,\CC} = \bigoplus_k \(G_{X,\CC}^{(k),+} \oplus G_{X,\CC}^{(k),-}\).$$
For instance, an element of $G_{X,\CC}^{(k),+}$ is an endomorphism of $T_{X,\CC}^+ \cap H^{(k)}(X).$
\end{rem}

\begin{nt}
More generally, let $R^*_K:=K[[(U_k)]]$ be a graded $K$-algebra with $K$ a field and $(U_k)_{1 \leq k \leq N}$ a set of formal variables.
Let $V$ be a $K$-vector space, $V_{R^*_K}$ be its base-change to $R^*_K$, and $f \colon V_{R^*_K} \to V_{R^*_K}$ be a $R^*_K$-linear map.
The $K$-algebra of all coefficients of $f$ with respect to the formal variables $(U_k)$ is denoted by
$$K[f]_{R^*_K} := K\left[\mathrm{Coeff}_{U_1^{a_1}\dotsm U_N^{a_N}}\(f\)\right]_{a_1, \dotsc, a_N \in \NN}.$$
In particular, we have $G_{X,\QQ} = \QQ[\kappa^{X,\mathrm{trans}}_\tau]_{R^*(X,\QQ)}$.
\end{nt}

\begin{pro}\label{comp}
Let $X$ be either a surface, or a fourfold with no degree-two transcendantal classes.
The $\QQ$-algebra $G_{X,\QQ}^+$ coincides with the $\QQ$-algebra defined in \cite[Definition 3.2]{OpenAI}.
\end{pro}

\begin{proof}
Fix $\underline{a}:=(a_0,\dotsc,a_h) \in \NN^{1+h}$ and $\beta \in \mathrm{NE}_\NN(X)$.
Using the divisor equation, the coefficient of $T_0^{a_0} \dotsm T_h^{a_h} Q^\beta$ in $\(\kappa_\tau^{X,\mathrm{trans}}\)_{|T^+_{X,\QQ}}$ equals
$$\cfrac{\int_{\beta} c_1(X) + \sum_{k=0}^h \(1-\frac{\deg(\alpha_k)}{2}\)a_k}{a_0! \dotsm a_h!} ~~\Phi^+_{\underline{a},\beta},$$
where the endomorphism $\Phi^+_{\underline{a},\beta} \colon T^+_{X,\QQ} \to T^+_{X,\QQ}$ is defined by
$$\Phi^+_{\underline{a},\beta}(u) = \sum_{j \in J^+} \langle u, e^j, \alpha_0^{a_0}, \dotsc, \alpha_h^{a_h} \rangle_{0,2+a_0+\dotsb+a_h,\beta} ~ e_j$$
on $u \in T^+_{X,\QQ}$, where the homogeneous basis $(e_0, \dotsc, e_m)$ of $H^*(X,\QQ)$ respects the decompositions $H^*(X,\QQ)_\Hdg \oplus T_{X,\QQ}$ and into even/odd degrees, and $J^+$ is the subset of indices for which $e_j \in T^+_{X,\QQ}$.
Observe that, for the correlators in $\Phi^+_{\underline{a},\beta}$ to be non-zero, the dimension equation implies
$$\cfrac{\int_{\beta} c_1(X) + \sum_{k=0}^h \(1-\frac{\deg(\alpha_k)}{2}\)a_k}{a_0! \dotsm a_h!}=\cfrac{2+\deg(u)-\deg(e_j)}{2 \cdot a_0! \dotsm a_h!},$$
so that the coefficient of $\Phi^+_{\underline{a},\beta}$ vanishes if and only if $\deg(e_j)=\deg(u)+2$.
Since $T_{X,\QQ}^+$ is pure-degree, then the relevant coefficients are always non-zero, so that
$$G_{X,\QQ}^+=\QQ[\(\Phi_{\underline{a},\beta}\)_{|T_{X,\QQ}^+}].$$

Definition 3.2 in \cite{OpenAI} provides the $\QQ$-algebra generated by the endomorphisms $\Phi^+_{\underline{i}, \underline{r},\beta} \colon T_{X,\QQ}^+ \to T_{X,\QQ}^+$
defined by
$$\Phi^+_{\underline{i}, \underline{r},\beta}(u) = \sum_{j \in J^+} \langle \tau_{r_{n+1}}(u), \tau_{r_{n+2}}(e^j), \tau_{r_1}(\alpha_{i_1}), \dotsc, \tau_{r_n}(\alpha_{i_n}) \rangle_{0,2n,\beta} ~ e_j,$$
using descendants (i.e.~$\psi$-classes), with $n \in \NN$, $\underline{i} \in \left\lbrace 0, \dotsc, h \right\rbrace^n$, and $\underline{r} \in \NN^{n+2}$.
Since $\psi$-classes in genus zero can be rewritten in terms of boundary divisors, correlators with descendants as above can be rewritten using the gluing axiom in terms of products of correlators without descendants (i.e.~with $\underline{r}=0$).
These are exactly the endomorphisms $\Phi^+_{\underline{a},\beta}$.
\end{proof}

\subsection{Blow-up formula}
Let $i \colon Z \hookrightarrow Y$ be an embedding of varieties of codimension $r \geq 2$. Denote by $X$ the blow-up of $Y$ at $Z$ and by $\mathrm{Bl} \colon X \to Y$ the blow-up map.
The classical decomposition of Hodge structures is
\begin{equation}\label{class}
H^*(X,\ZZ) \simeq H^*(Y,\ZZ) \oplus \bigoplus_{k=1}^{r-1} H^{*-2k}(Z,\ZZ)(-k),
\end{equation}
given by $\mathrm{Bl}^*$ on $Y$ and by $j_*\circ \(\xi^{k-1} \cup \) \circ \mathrm{Bl}_{|E}^*$ on the summands, where $j \colon E \hookrightarrow X$ is the embedding of the exceptional divisor and $\xi$ is its hyperplane class.
It is an isomorphism of $\ZZ$-Hodge structures.
Furthermore, it is an isometry when $r=2$.

Following Iritani \cite{Iritani},  we define the integer
$$s := \left\lbrace \begin{array}{ll}
2(r-1) & \textrm{if $r$ is odd,} \\
r-1 & \textrm{if $r$ is even,}
\end{array}\right.$$
and the graded $\QQ$-algebra
$$R^*_\mathrm{Bl} := \QQ[\mathfrak{q}^{\pm 1/s}][[Q_Y,T_0,\dotsc,T_h]],$$
where $Q_Y$ is the Novikov variable of $Y$, $T_0, \dotsc, T_h$ are the formal coordinates associated to a (homogeneous) basis of $H^*(X,\QQ)_\Hdg$ as before, and $\mathfrak{q}$ is a new formal variable of degree $2(r-1)$.
Recall that we have three maps of graded $\QQ$-algebras
$$R^*(X,\QQ) \hookrightarrow R^*_\mathrm{Bl}~,~~ 
R^*(Y,\QQ) \hookrightarrow R^*_\mathrm{Bl}~,~~R^*(Z,\QQ) \to R^*_\mathrm{Bl}$$
defined as the identity on formal coordinates (using the above classical decomposition of cohomology) and by the following on the Novikov variables:
$$Q_X^{\beta_X} \mapsto Q_Y^{\mathrm{Bl}_*(\beta_X)} \mathfrak{q}^{-[E] \cdot \beta_X}~,~~Q_Y \mapsto Q_Y~,~~Q_Z^{\beta_Z} \mapsto Q_Y^{i_*(\beta_Z)} \mathfrak{q}^{-\frac{\int_{\beta_Z} c_1(\cN_i)}{r-1}},$$
where $\cN_i$ is the normal bundle of the embedding $i$.
In particular, we have
$$Q_X^{l_E} \mapsto \mathfrak{q},$$
where $l_E$ is the `vertical' line in the exceptional divisor $E$ defined by $l_E = \xi^{r-2}$.

\begin{nt}
Let $A \in \left\lbrace X,Y,Z\right\rbrace$ and $K$ be a field extension of $\QQ$.
We use the notation
$$\widehat{G}_{A,K}:= K[\kappa^{X,\mathrm{trans}}_\tau]_{R^*_{\mathrm{Bl},K}},$$
that is the $K$-algebra of the coefficients of $\kappa^{A,\mathrm{trans}}_\tau$ with respect to the formal variables $\mathfrak{q}^{1/s}$, $Q_Y$, $T_0, \dotsc, T_h$.
Note that we have
$$\widehat{G}_{A,K} = \widehat{G}_{A,\QQ} \otimes_\QQ K.$$
\end{nt}

\begin{lem}\label{compl}
We have
$$\widehat{G}_{X,K}=G_{X,K} \textrm{ and } \widehat{G}_{Y,K}=G_{Y,K}.$$
\end{lem}

\begin{proof}
It follows from the injectivity of $R^*(X,\QQ) \hookrightarrow R^*_\mathrm{Bl}$ and $R^*(Y,\QQ) \hookrightarrow R^*_\mathrm{Bl}$.
\end{proof}

\begin{lem}\label{surface}
Let $Z$ be a surface.
Then we have
$$\widehat{G}_{Z,\QQ}^+=G_{Z,\QQ}^+= \QQ \cdot \id_{T_{Z,\QQ}^+}.$$
Moreover, if $K_Z$ is nef, then $\widehat{G}^-_{Z,\QQ}=G^-_{Z,\QQ}$ and it is generated by the identity on $T^-_{Z,\QQ}$ and cup products with degree-$2$ Hogde classes.
\end{lem}

\begin{proof}
Fix $\underline{a}:=(a_0,\dotsc,a_h) \in \NN^{1+h}$ and $\beta \in \mathrm{NE}_\NN(Z)$.
Then the coefficient of $T_0^{a_0} \dotsm T_h^{a_h} Q_X^\beta$ in $\kappa_\tau^{Z,\mathrm{trans}}$ equals
$$\cfrac{\int_{\beta} c_1(Z) + \sum_{k=0}^h \(1-\frac{\deg(\alpha_k)}{2}\)a_k}{a_0! \dotsm a_h!} ~~ \Phi_{\underline{a},\beta},$$
where the endomorphism $\Phi_{\underline{a},\beta} \colon T_{Z,\QQ} \to T_{Z,\QQ}$ is defined by
$$\Phi_{\underline{a},\beta}(u) = \sum_{j \in J} \langle u, e^j, \alpha_0^{a_0}, \dotsc, \alpha_h^{a_h} \rangle_{0,2+a_0+\dotsb+a_h,\beta} ~ e_j$$
on $u \in T_{Z,\QQ}$, where the homogeneous basis $(e_0, \dotsc, e_m)$ of $H^*(Z,\QQ)$ respects the decompositions $H^*(Z,\QQ)_\Hdg \oplus T_{Z,\QQ}$ and into even/odd degrees, and $J$ is the subset of indices for which $e_j \in T_{Z,\QQ}$.

Since $Z$ is a surface, we can use the string, divisor, and degree-$1$ equations to eliminate most of the correlators.
Precisely, if $\underline{a}=(\delta_{k=i})_{0 \leq k \leq h}$ for some $i$ and $\beta=0$, then
\begin{eqnarray*}
\Phi_{\underline{a},\beta}(u) & = & \sum_{j \in J} \langle u, e^j, \alpha_i \rangle_{0,3,0} ~ e_j \\
& = & \sum_{j \in J} \int_Z u\cup e^j\cup \alpha_i ~ e_j \\
& = & u \cup \alpha_i \\
& = & \left\lbrace \begin{array}{cl}
u & \textrm{if $\alpha_i \in H^0(Z)$,} \\
u \cup \alpha_i & \textrm{if $\alpha_i \in H^2(Z)$ and $u \in H^1(Z)$,} \\
0 & \textrm{otherwise.}
\end{array}\right. 
\end{eqnarray*}
Otherwise, we have
$$\Phi_{\underline{a},\beta}(u) = \left\lbrace \begin{array}{cc}
0 & \textrm{if $u \in H^{\leq 2}(Z)$,} \\
\(\delta_{a_0=0} \prod_{k=1}^{h-1} \(\int_{\beta}\alpha_k\)^{a_k}\) \cdot \sum_{j \in J} \langle u, e^j, \alpha_h^{a_h} \rangle_{0,2+a_h,\beta} ~ e_j & \textrm{if $u \in H^3(Z)$,}
\end{array}
\right.
$$
where we use $\int_\beta u=0$ and we choose $\alpha_0=\un_Z$, $\alpha_1, \dotsc, \alpha_{h-1} \in H^2(Z)$, and $\alpha_h \in H^4(Z)$.
Note that the second line is non-zero only if $e_j \in H^1(Z)$ and $\int_{\beta}c_1(Z)=2+a_h$.
In particular, it is zero when $K_Z$ is nef, but it could be non-zero in general.
As a consequence, when restricted to $T_{Z,\QQ}^+$, we obtain
$$G_{Z,\QQ}^+=\QQ \cdot \id_{T_{Z,\QQ}^+}.$$
Since this is achieved by a correlator with $\beta=0$, then we also find
$$\widehat{G}_{Z,\QQ}^+=\QQ \cdot \id_{T_{Z,\QQ}^+},$$
as well as the result on $G_{Z,\QQ}^-$ when $K_Z$ is nef.
\end{proof}

\begin{lem}\label{curve}
Let $Z$ be a curve.
Then we have
$$\widehat{G}_{Z,\QQ}=G_{Z,\QQ}= \QQ \cdot \id_{T_{Z,\QQ}}.$$
\end{lem}

\begin{proof}
As in the previous proof, using the string, divisor, and degree-one equations, the only non-zero $\Phi_{\underline{a},\beta}$ is the identity.
\end{proof}

%

At last, we recall Iritani's blow-up formula in GW theory.
\begin{thm}[Iritani, \cite{JG,Iritani}]\label{Bl}
Let $K$ be the number field obtained by adding the $s$-th roots of unity and the imaginary unit $\ci$.
Then there exist an invertible morphism of graded $K$-algebras of degree zero $g \colon R^*_{\mathrm{Bl},K} \to R^*_{\mathrm{Bl},K}$ and a $R^*_{\mathrm{Bl},K}$-linear isomorphism
$$\Psi ~~\colon \(H^*(Y,K) \oplus \bigoplus_{k=1}^{r-1} H^{*-2k}(Z,K)\)_{R^*_{\mathrm{Bl},K}} \to H^*\(X,K\)_{R^*_{\mathrm{Bl},K}}$$
such that
\begin{eqnarray*}
\(\kappa^Y \oplus \bigoplus_{k=1}^{r-1} \kappa^Z \)_\tau & = & \Psi^{-1} \circ \kappa^X_{g(\tau)} \circ \Psi,
\end{eqnarray*}
where $\tau :=T_0 \alpha_0+\dotsb+T_h \alpha_h$.
Moreover, the isomorphism $\Psi$ is in the $K$-span of morphisms of $\QQ$-Hodge structures.
\end{thm}

\begin{rem}
We will see in the next subsection that the isomorphism $\Psi$ respects the decomposition into even/odd degrees in cohomology.
Therefore, if $S$ is a surface with minimal model $\Sigma$, then
$$G_{S,K}^+=G_{\Sigma,K}^+.$$
Moreover, when $K_\Sigma$ is nef, the algebra $G^-_{S,K}=G_\Sigma$ is generated by the identity and cup products with degree-$2$ Hodge classes.
\end{rem}

\section{Refinement}
We investigate Iritani's blow-up formula in order to show that it is compatible with the usual blow-up decomposition in some cases.

\begin{lem}
Let $i \colon U \to V$ be a morphism of varieties and $\alpha \in H^*(U,\QQ)_\Hdg$ be a Hodge class.
Then the maps $i^*$, $i_*$, and $\alpha \cup$ preserve the decompositions into Hodge/transcendantal classes and into even/odd degrees.
This means e.g.~that the maps
\begin{eqnarray*}
i_* & \colon & H^*(U,\QQ)_\Hdg \oplus T_{U,\QQ} \to H^*(V,\QQ)_\Hdg \oplus T_{V,\QQ} \\
i_* & \colon & H^\even(U,\QQ) \oplus H^\mathrm{odd}(U,\QQ) \to H^\even(V,\QQ) \oplus H^\mathrm{odd}(V,\QQ)
\end{eqnarray*}
are diagonal by blocks (and similarly for the two other maps).
Furthermore, their complexifications preserve the Hochschild degree.
\end{lem}

\begin{pro}
The isomorphism $\Psi$ respects the decompositions into Hodge/transcendantal classes and into even/odd degrees.
Furthermore, its complexification respects the Hochschild degree.
\end{pro}

\begin{proof}
The only ingredients used in the construction of $\Psi$ are indeed pull-backs, push-forwards, and cup product with algebraic classes such as the virtual fundamental cycle or the Chern classes of the normal bundle of $Z$ in $Y$.
\end{proof}

Consequently, we obtain an isomorphism
$$\Phi ~~\colon T_{Y,R^*_{\mathrm{Bl},K}} \oplus \bigoplus_{k=1}^{r-1} T_{Z,R^*_{\mathrm{Bl},K}}(-k) \to T_{X,R^*_{\mathrm{Bl},K}}$$
by
taking the restriction of $\Psi$ to the transcendantal blocks.
Note that it is still in the $K$-linear span of morphisms of $\QQ$-Hodge structures.
Furthermore, we have the conjugation relation
\begin{eqnarray*}
\(\kappa^{Y,\mathrm{trans}} \oplus \bigoplus_{k=1}^{r-1} \kappa^{Z,\mathrm{trans}} \)_\tau & = & \Phi^{-1} \circ \kappa^{X,\mathrm{trans}}_{g(\tau)} \circ \Phi.
\end{eqnarray*}

\begin{pro}
We have an equality of $K$-algebras
$$G_{X,K} = \Phi \circ \(G_{Y,K} \oplus \bigoplus_{k=1}^{r-1} \widehat{G}_{Z,K}\) \circ \Phi^{-1}.$$
\end{pro}

\begin{proof}
First, we recall that
$$\widehat{G}_{X,K}=G_{X,K} \textrm{ and } \widehat{G}_{Y,K}=G_{Y,K}.$$
by Lemma \ref{compl}.
Second, the above conjugation relation yields
$$K[\kappa^{X,\mathrm{trans}}_{g(\tau)}]_{R^*_{\mathrm{Bl},K}} = \Phi \circ \(K[\kappa^{Y,\mathrm{trans}}_{\tau}]_{R^*_{\mathrm{Bl},K}} \oplus \bigoplus_{k=1}^{r-1} K[\kappa^{Z,\mathrm{trans}}_{\tau}]_{R^*_{\mathrm{Bl},K}}\) \circ \Phi^{-1}.$$
Hence, it remains to prove the equality
$$K[\kappa^{X,\mathrm{trans}}_{g(\tau)}]_{R^*_{\mathrm{Bl},K}}=K[\kappa^{X,\mathrm{trans}}_{\tau}]_{R^*_{\mathrm{Bl},K}},$$
which is indeed a general fact regarding a change of variables.
%
\end{proof}

We recall from \cite[Section 5.8]{Iritani} that the isomorphism $\Psi$, and thus $\Phi$, can be reconstructed from initial datas.
Specifically, we denote by
$$\Phi^\circ:=\Phi_{|Q=\tau=0} \in \mathrm{Hom}(T_{Y,K} \oplus \bigoplus_{k=1}^{r-1}T_{Z,K}(-k),T_{X,K})[z]((\mathfrak{q}^{-1/s}))$$
the initial condition and by
$$M \in \mathrm{End}(H^*(Y,\QQ) \oplus \bigoplus_{k=1}^{r-1} H^{*-2k}(Z,\QQ))[z^{-1}]\otimes_\QQ R^*_{\mathrm{Bl},\QQ}$$
the fundamental solution written in \cite[Page 61]{Iritani} and relying on the fundamental solutions of the genus-zero GW theories of $Y$ and of $Z$.
Note that it is diagonal by blocks with respect to the decomposition
$$H^*(Y,\QQ) \oplus \bigoplus_{k=1}^{r-1} H^{*-2k}(Z,\QQ)$$
and that it also respects the decomposition into Hodge/transcendantal classes.
Note also that a classical property of the fundamental solution in GW theory is
$$M = \id + O(z^{-1}).$$
We denote by
$$M^\mathrm{trans} \in \mathrm{End}(T_{Y,\QQ} \oplus \bigoplus_{k=1}^{r-1}T_{Z,\QQ}(-k))[z^{-1}]\otimes_\QQ R^*_{\mathrm{Bl},\QQ}$$
its restriction to the transcendantal parts.
We then set
$$L:= \(\Phi^\circ\)^{-1} M^\mathrm{trans} \in \mathrm{End}(T_{Y,K} \oplus \bigoplus_{k=1}^{r-1}T_{Z,K}(-k))[z,z^{-1}]\otimes_K R^*_{\mathrm{Bl},K},$$
where we use the classical isomorphism
$$T_{X,R^*_{\mathrm{Bl},K}} \to T_{Y,R^*_{\mathrm{Bl},K}} \oplus \bigoplus_{k=1}^{r-1} T_{Z,R^*_{\mathrm{Bl},K}}(-k)$$
inherited from Equation \eqref{class} and tensored with $R^*_{\mathrm{Bl},K}$.

The Birkhoff factorization enables us to write uniquely the operator $L$ as
$$L=:L_- \circ L_+$$
where $L_+$ corresponds to non-negative exponents in the variable $z$.
At last, we recover the isomorphism $\Phi$ as the non-negative part
$$\Phi = L_+.$$

\begin{lem}\label{lem1}
If the initial data $\Phi^\circ$ is diagonal by blocks with respect to the decomposition
$$T_{Y,K} \oplus \bigoplus_{k=1}^{r-1}T_{Z,K}(-k)$$
(once we identify $T_{X,K}$ with $T_{Y,K} \oplus \bigoplus_{k=1}^{r-1}T_{Z,K}(-k)$ by Equation \eqref{class}), then the same property holds for the isomorphism $\Phi$.
\end{lem}

\begin{proof}
It follows from the Birkhoff factorization because the matrix $M^\mathrm{trans}$ is also diagonal by blocks with respect to the decomposition
$T_{Y,\QQ} \oplus \bigoplus_{k=1}^{r-1}T_{Z,\QQ}(-k).$
\end{proof}

\begin{lem}\label{lem2}
Assume that the initial data $\Phi^\circ$ is diagonal by blocks as in the previous lemma and furthermore that it acts by scalars
$$\lambda_Y, \lambda_{Z,1}, \dotsc, \lambda_{Z,r-1} \in K((\mathfrak{q}^{-1/s}))$$
on the respective blocks.
Then we have
$$\Phi = \Phi^\circ.$$
In particular, it acts by scalars on each blocks.
\end{lem}

\begin{proof}
It is again a consequence of the Birkhoff factorization.
Indeed, on the block $T_{Y,K}$, we have
$$L_{|T_{Y,K}} = \lambda_Y^{-1} M_{|T_{Y,K}} = \lambda_Y^{-1} \id + O(z^{-1})$$
so that the non-negative part in $z$ is $\Phi^{-1} = \lambda_Y^{-1} \id$.
The same happens on the other blocks $T_{Z,K}(-k)$.
\end{proof}

We show that, under some condition, the assumptions of Lemmas \ref{lem1} and \ref{lem2} are indeed satisfied.
\begin{pro}\label{initial}
Let $\epsilon \in \left\lbrace +,- \right\rbrace$.
We assume that
\begin{itemize}
\item $T_{Y,\QQ}^\epsilon$ and $T_{Z,\QQ}^\epsilon$ lie in their respective middle cohomologies.
\end{itemize}
In particular, the restriction of $T_{Y,\QQ}^\epsilon$ to $Z$ vanishes.
Then the initial condition $\Phi^\circ$ is diagonal by blocks with respect to the two blocks in the decomposition
$$T_{Y,K}^\epsilon \oplus \(\bigoplus_{k=1}^{r-1}T_{Z,K}^\epsilon(-k)\)$$
and it acts as the identity on the first block $T_{Y,K}^\epsilon$.
Furthermore, if $r=2$, then the initial condition $\Phi^\circ$ acts by the imaginary unit $\ci$ on the second block $T_{Z,K}^\epsilon(-1)$.
\end{pro}

\begin{proof}
We use notations from \cite{Iritani}.
According to \cite[Equation 5.44,5.46]{Iritani}, written in the basis
$$\(\mathrm{pr}^*(\phi_{Y,i}),(j_*(\xi^l\mathrm{pr}_{|E}^*\phi_{Z,m}))_{0 \leq l \leq r-2}\)$$
of $T_{X,\QQ}^\epsilon$, we have the matrix
$$\Phi^\circ = \begin{bmatrix}
A & B \\
C_j & D_j \\
\end{bmatrix},$$
with
\begin{eqnarray*}
A(\mathrm{pr}^*(\phi_{Y,i})) & = & e^{-\tau^\circ/z} \phi_{Y,i} + e^{-\tau^\circ/z} \sum_{k>0} \mathfrak{q}^{-k} i_*\(\cfrac{\prod_{\nu=1}^{k-1}e_{\nu z}(\cN_i)}{k!z^k}\) \cup \phi_{Y,i}, \\
B(j_*(\xi^l\mathrm{pr}_{|E}^*\phi_{Z,m})) & = & e^{-\tau^\circ/z} \sum_{k>0} \mathfrak{q}^{-k} i_*\(\cfrac{\prod_{\nu=1}^{k-1}e_{\nu z}(\cN_i)}{k!z^k}(-1)^l(kz)^{l+1}\phi_{Z,m}\), \\
C_j(\mathrm{pr}^*(\phi_{Y,i})) & = & e^{-((r-1)\lambda_j+\varsigma^\circ_j)/z}\mathfrak{q}^{\rho_Z/((r-1)z)} \cF_{Z,j}(c_i), \\
D_j(j_*(\xi^l\mathrm{pr}_{|E}^*\phi_{Z,m})) & = & e^{-((r-1)\lambda_j+\varsigma^\circ_j)/z}\mathfrak{q}^{\rho_Z/((r-1)z)} \cF_{Z,j}(c_{l,m}).
\end{eqnarray*}
Since $T_{Y,\QQ}^\epsilon$ is pure-degree, then $x \cup \phi_{Y,i}$ vanishes for any positive-degree class $x$, so that we get
\begin{eqnarray*}
A & = & e^{-\tau^\circ/z} \id \\
& = & \exp\(-\frac{1}{z} \cdot [z^{-1}] \log \(1+\sum_{k>0} \mathfrak{q}^{-k} i_*\(\cfrac{\prod_{\nu=1}^{k-1}e_{\nu z}(\cN_i)}{k!z^k}\)\)\) \id \\
& = & \id.
\end{eqnarray*}
Recall the continuous Fourier transform from \cite[Equation 4.7]{Iritani}:
$$\mathbf{f} \to \int e^{\lambda \log S_Z/z} G_Z \mathbf{f}_{|Z} d\lambda.$$
This is the Mellin--Barnes integral and the operator $\cF_{Z,j}$ is then deduced from it.
Since this integral uses the restriction of $\mathbf{f}$ to $Z$ and since the restriction of $T_{Y,\QQ}^\epsilon$ to $Z$ is zero, then we get $\cF_{Z,j}(c_i)=0$ and thus
$$C_j=0.$$
Since $T_{Y,\QQ}^\epsilon$ restricted to $Z$ is zero, then we have $i_*(\phi_{Z,m})=0$.
Since $T_{Z,\QQ}^\epsilon$ is pure-degree, then $x \cup \phi_{Z,m}$ vanishes for any positive-degree class $x$.
Therefore, we get
$$B=0.$$
So far, we have proven
$$\Phi^\circ = \begin{bmatrix}
\id & 0 \\
0 & D_j \\
\end{bmatrix},$$
which is the first part of Proposition \ref{initial}.
Now, we assume that $r=2$ and we aim at proving that $D=\ci \cdot \id$.
We have
\begin{eqnarray*}
D(j_*(\mathrm{pr}_{|E}^*\phi_{Z,m})) & = & e^{-(\lambda_0+\varsigma^\circ_0)/z}\mathfrak{q}^{\rho_Z/z} \cF_{Z,0}(c_{0,m}) \\
& = & e^{-\sigma^\circ_0/z}\exp\(\frac{\log \mathfrak{q}}{z} c_1(\cN_i)\) \cF_{Z,0}(c_{0,m}),
\end{eqnarray*}
because $\lambda_0+\varsigma^\circ_0=\sigma^\circ_0$ and $\rho_Z = c_1(\cN_i)$.
Observe that $\sigma^\circ_0$ has no component of cohomological degree zero.
Indeed, by \cite[Equation 5.17]{Iritani}, we have
$$\sigma^\circ_0 = h_{Z,0} + O(\mathfrak{q}^{-1}),$$
with $h_{Z,0}=\ci \pi \rho_Z$ a degree-two class.
Since it is homogeneous of degree two and $\deg(\mathfrak{q}^{-1})=-2$, then the term $O(\mathfrak{q}^{-1})$ must involved cohomology classes of positive degrees.
Since the cup product of a class of positive degree with a class of $T^\epsilon_{Z,\QQ}$ vanishes, then we deduce
$$D(j_*(\mathrm{pr}_{|E}^*\phi_{Z,m})) = \cF_{Z,0}(c_{0,m}).$$
We also obtain
\begin{eqnarray*}
G_Z & = & \cfrac{(-z)^{-\lambda/z}}{(\sqrt{-2\pi z})^3} \Gamma(-\lambda/z) \prod_{\substack{\epsilon \\ \textrm{root of $\cN_i$}}} (-z)^{(\lambda-\epsilon)/z} \Gamma\(\frac{\lambda-\epsilon}{z}\) \\
& = & (-z)^{-\lambda/z}(-2\pi z)^{-3/2} \Gamma(-\lambda/z)  \((-z)^{\lambda/z} \Gamma(\lambda/z)\)^2 \\
& = & (2\pi)^{-3/2} (-z)^{\lambda/z-3/2}  \Gamma(\lambda/z)^2 \Gamma(-\lambda/z)
\end{eqnarray*}
on $T^\epsilon_{Z,\QQ}$, because all non-equivariant Chern roots of the normal bundle $\cN_i$ vanish once we take the cup product with a class in $T^\epsilon_{Z,\QQ}$.
Recall \cite[Equation 4.9]{Iritani}
$$\int e^{\lambda \frac{\log S_Z}{z}} G_Z \cdot  (c_{0,m})_{|Z} d\lambda \sim \sqrt{2 \pi z} \cdot e^{c_Z \lambda_0/z} \cF_{Z,0}(c_{0,m}),$$
where the equivalence means the formal asymptotic expansion as $z \to 0$.
We have $S_Z = \mathfrak{q}^{-1}$ and $(c_{0,m})_{|Z}=\lambda \cdot \phi_{Z,m}$.

Set $x=\lambda/z$ and $y=-z/\mathfrak{q}$.
In particular, we get
$$e^{\lambda \frac{\log S_Z}{z}} = \mathfrak{q}^{-x}.$$
We obtain
$$\int e^{\lambda \log S_Z /z} G_Z \cdot  (c_{0,m})_{|Z} d\lambda = \frac{\ci}{(2 \pi)^{3/2}} \cdot z^{1/2} \int y^x x \Gamma(x)^2 \Gamma(-x) dx \cdot \phi_{Z,m}.$$
Introduce the functions $\psi(x):=\Gamma(x)^2 \Gamma(-x)$ and $h(y)=\int y^x \psi(x) dx$.
Observe that
$$-(x+1)\psi(x+1)=x^2\psi(x).$$
Introduce the differential operator $\partial := y \frac{d}{dy}$.
Then we have
$$\partial^2 h = \int y^x x^2 \psi(x) dx = - y^{-1} \partial h.$$
Set $k := \partial h$.
It is a solution of the differential equation
$$y \frac{d}{dy}(k) = -\frac{k}{y}$$
so that
$$k = \mathrm{Cst} e^{1/y} = \mathrm{Cst}~ e^{-\mathfrak{q}/z}.$$
Therefore, we obtain
\begin{eqnarray*}
\sqrt{2 \pi z} \cdot e^{c_Z \lambda_0/z} \cF_{Z,0}(c_{0,m}) & \sim & \frac{\ci}{(2 \pi)^{3/2}} \cdot z^{1/2} k(y) \cdot \phi_{Z,m} \\
& \sim & \frac{\ci}{(2 \pi)^{3/2}} \cdot z^{1/2} \mathrm{Cst} ~ e^{-\mathfrak{q}/z} \cdot \phi_{Z,m}.
\end{eqnarray*}
Since we have $\sqrt{2 \pi z} \cdot e^{c_Z \lambda_0/z} = \sqrt{2 \pi z} \cdot e^{-\mathfrak{q}/z}$, then we obtain
$$\cF_{Z,0}(c_{0,m}) \sim \frac{\ci}{4 \pi^2} \cdot \mathrm{Cst} \cdot \phi_{Z,m}.$$
Consequently, the matrix $D$ is a multiplication by a constant, which equals the dominant term in \cite[Equation 4.10]{Iritani}, that is $q_{Z,0} \lambda_0 = \ci$, the imaginary unit.
Hence, we get
$$\Phi^\circ = \begin{bmatrix}
\id & 0 \\
0 & \ci \cdot \id \\
\end{bmatrix},$$
finishing the proof.
\end{proof}


\begin{cor}\label{cordec}
Let $\epsilon \in \left\lbrace +,- \right\rbrace$.
We assume that
\begin{itemize}
\item $T_{Y,\QQ}^\epsilon$ and $T_{Z,\QQ}^\epsilon$ lie in their respective middle cohomologies.
\end{itemize}
Under the classical decomposition $T_X^\epsilon = T_Y^\epsilon \oplus \bigoplus_{k=1}^{r-1} T_Z^\epsilon(-k)$, we have the equality
$$G_{X,K}^\epsilon = G_{Y,K}^\epsilon \oplus \(\Phi \circ \(\bigoplus_{k=1}^{r-1} {\widehat{G}_{Z,K}}^\epsilon\) \circ \Phi^{-1}\).$$
Furthermore, if $r=2$, then we have
$$G_{X,K}^\epsilon = G_{Y,K}^\epsilon \oplus \widehat{G}_{Z,K}^\epsilon.$$
\end{cor}

\begin{rem}
Together with Lemma \ref{surface} and Proposition \ref{comp}, this corollary implies \cite[Theorem 3.3]{OpenAI}.
\end{rem}


\section{Birational geometry of fourfolds}
\subsection{Decomposition of the algebra}
We start with a structural lemma on $K$-algebras.

\begin{lem}\label{decomp}
Let $G$ be a $K$-algebra of finite dimension.
We denote by $\mathcal{S}_G \subset G$ the set of all non-zero elements $x \in G$ such that:
\begin{itemize}
\item $x^2=x$,
\item for any $y \in G$, we have $xy \in K \langle x \rangle$ and $yx \in K \langle x \rangle$.
\end{itemize}
Then the set $\mathcal{S}_G$ is finite, and denoting by $x_1, \dotsc, x_n$ its elements, they are linearly independant.
Furthermore, there is a decomposition of $K$-algebras
$$G = H \oplus \bigoplus_{k=1}^n \langle x_k \rangle,$$
where $H := \bigcap_{k=1}^n \ker(x_k \cdot (-))$, and such a decomposition is unique.
\end{lem}

\begin{proof}
Let $x,x' \in \mathcal{S}_G$ such that $x \cdot x' \neq 0$. By definition, we have
$$x \cdot x' \in K \langle x \rangle \cap K \langle x' \rangle,$$
so that $x'=\lambda x$ for some $\lambda \in K^*$. Since $x^2=x$ and ${x'}^2=x'$, we get $\lambda=1$, so that $x=x'$.
Consequently, we obtain
$$\forall ~x \neq x' \in \mathcal{S}_G~,~~x\cdot x' = 0.$$
If we have a zero linear combination
$$\sum_k a_k x_k =0,$$
with $x_k \in \mathcal{S}_G$ and $a_k \in K$, then multiplying with $x_i$ gives $a_i=0$, so that elements of $\mathcal{S}_G$ are linearly independant and thus $\mathcal{S}_G$ is finite.
	
Define $H$ as the intersection of kernels of multiplications on the left by $x_1, \dotsc, x_n$.
Note that, for any $y \in H$ and any $1 \leq k \leq n$, there exists $\lambda \in K$ such that $y x_k = \lambda x_k$.
Hence, we have $y x_k = \lambda x_k^2 = x_k y x_k = 0$, so that $H$ is also the intersection of kernels of right multiplications by $x_1, \dotsc, x_n$.

Since all $x_k$ are projectors, we have a direct sum $G = H \oplus \bigoplus_{k=1}^n \langle x_k \rangle$ as $K$-vector spaces.
It remains to show that $H$ is stable under multiplication.

Let $y_1, y_2 \in H$ and write $y_1 y_2 = y_3 + a_1 x_1 \dotsb + a_n x_n$ with $a_k \in K$ and $y_3 \in H$.
Then for each $1 \leq k \leq n$ we have
$$0=(x_k y_1) y_2 = x_k (y_1y_2) = a_k x_k,$$
so that $a_1=\dotsc=a_n=0$ and $y_1y_2 \in H$.

Conversely, if we have a decomposition of $K$-algebras $G = H \oplus \bigoplus_{k=1}^n \langle x_k \rangle$, then for any $1 \leq k \leq n$ and for any $y \in H$, we have $x_k y = 0$, so that $H$ is indeed the intersection of kernels as in the statement.
\end{proof}

\begin{pro}\label{sum}
Let $G_1$ and $G_2$ be two $K$-algebras, whose decompositions given by Lemma \ref{decomp} are
$$G_1 = H_1 \oplus \bigoplus_{k=1}^{n_1} \langle x_k \rangle \textrm{ and } G_2 = H_2 \oplus \bigoplus_{k=1}^{n_2} \langle y_k \rangle.$$
Then the decomposition of $G_1 \oplus G_2$ is
$$G_1 \oplus G_2 = \(H_1 \oplus H_2\) \oplus \(\bigoplus_{k=1}^{n_1} \langle x_k \rangle \oplus \bigoplus_{k=1}^{n_2} \langle y_k \rangle\).$$
\end{pro}

\begin{proof}
First, for each $k$, we have $x_k, y_k \in \mathcal{S}_{G_1 \oplus G_2}$.
Second, the direct sum $H_1 \oplus H_2$ equals the intersection of kernels of the left multiplication by all $x_k$'s and $y_k$'s.
Third, the intersection $\(H_1 \oplus H_2\) \cap \mathcal{S}_{G_1 \oplus G_2}$ is empty.
Indeed, otherwise we would have $z_1 \in H_1$ and $z_2 \in H_2$ such that $z_1+z_2 \neq 0$, $(z_1+z_2)^2=z_1+z_2$ and for every $w_1 \in G_1$ and $w_2 \in G_2$ we would have
$$(z_1+z_2) (w_1+w_2) \in K\langle z_1+z_2 \rangle \textrm{ and } (w_1+w_2)(z_1+z_2) \in K\langle z_1+z_2 \rangle.$$
Assume for instance that $z_1 \neq 0$ (the case $z_2 \neq 0$ is identical).
We have $z_2^2=z_2$, and taking $w_1=0$ and $w_2=z_2$, we obtain $z_2=w_2z_2 \in K\langle z_1+z_2 \rangle$, which implies $z_2=0$.
Since $z_1 \notin \mathcal{S}_{G_1}$ and $z_1^2=z_1 \neq 0$, then there exists $w_1 \in G_1$ such that $z_1w_1 \notin K\langle z_1 \rangle$ or $w_1z_1 \notin K\langle z_1 \rangle$, which is a contradiction.
\end{proof}

\subsection{Integral structure}
\begin{dfn}
Let $X$ be a variety and $\epsilon \in \left\lbrace +,-,\emptyset \right\rbrace$.
We say that $G^\epsilon_{X,K}$ is \textit{integral} if the decomposition
$$G^\epsilon_{X,K} = H^\epsilon_{X,K} \oplus \bigoplus_{k=1}^{n_X} K\langle p^\epsilon_{X,k} \rangle$$
from Lemma \ref{decomp} is such that there exists a decomposition
$$T^\epsilon_X = U^\epsilon_X \oplus \bigoplus_{k=1}^{n_X} L^\epsilon_{X,k}$$
of free-$\ZZ$-modules with
$$\bigcap_{k=1}^{n_X} \ker(p^\epsilon_{X,k}) = U^\epsilon_X \otimes_\ZZ K \textrm{ and } \mathrm{Im} (p^\epsilon_{X,k}) = L^\epsilon_{X,k} \otimes_\ZZ K$$
for each $1 \leq k \leq n_X$.
Note that this decomposition into free-$\ZZ$-modules is unique.
\end{dfn}

\begin{pro}
Let $Y$ be a fourfold, $Z \subset Y$ be a subvariety, and $X$ be the blow-up of $Y$ at $Z$.
Then $X$ has no degree-two transcendantal classes if and only if $Y$ has no degree-two transcendantal classes.
Furthermore, in that case, $G^+_{X,K}$ is integral if and only if $G^+_{Y,K}$ is integral.
\end{pro}

\begin{proof}
The result on transcendantal classes follows from Equation \eqref{class}.

By Corollary \ref{cordec}, Lemma \ref{surface}, and Proposition \ref{sum}, we have
\begin{eqnarray*}
G^+_{X,K} & = & G^+_{Y,K} \oplus G^+_{Z,K}, \\
H^+_{X,K} & = & H^+_{Y,K}, \\
\mathcal{S}_{G^+_{X,K}} & = & \mathcal{S}_{G^+_{Y,K}} \cup \left\lbrace p^+_Z \right\rbrace,
\end{eqnarray*}
where $p^+_Z$ is the projector to $T^+_{Z,K}$ (if non-zero).
Furthermore, all these decompositions are compatible with Equation \eqref{class} tensored with $K$.
Consequently, $G^+_{X,K}$ is integral if and only if $G^+_{Y,K}$ is integral.
\end{proof}

Let $X$ and $Y$ be two birational fourfolds.
We assume that $Y$ has no degree-two transcendantal classes and $G^+_{Y,K}$ is integral.
Then the same holds for $X$.
Considering a weak factorization
$$X=:X_0 \leftrightarrow X_1 \leftrightarrow \dotsc \leftrightarrow X_q := Y,$$
with blow-up centers $Z_1, \dotsc, Z_q$, we obtain a natural isomorphism of $\ZZ$-Hodge structures
$$T^+_X \oplus \bigoplus_{k \in I_\leftarrow} T^+_{Z_k}(-1) \simeq T^+_Y \oplus \bigoplus_{k \in I_\rightarrow} T^+_{Z_k}(-1),$$
where $I_\leftarrow$ (resp.~$I_\rightarrow$) denotes the set of indices $j$ such that the map is in the direction $Z_{j-1} \leftarrow X_j$ (resp.~$Z_{j-1} \rightarrow X_j$).
Note that we can disregard the blow-up centers that are not surfaces, since $T^+_{Z_k}$ would then vanish.
Therefore, the codimension of the relevant centers is precisely $2$, so that the isomorphism is in fact an isometry.
Furthermore, we obtain the equality of $K$-algebras
$$G^+_{X,K} \oplus \bigoplus_{k \in I_\leftarrow} G^+_{Z_k,K} = G^+_{Y,K} \oplus \bigoplus_{k \in I_\rightarrow} G^+_{Z_k,K},$$
which is compatible with the previous decomposition and is integral.
Recall that $G^+_{Z_k,K} = K \langle p^+_{Z_k} \rangle$ is generated by the projector to $T^+_{Z_k,K}$.
Consequently, using the decomposition of Lemma \ref{decomp}, we obtain
\begin{eqnarray*}
H^+_{X,K} & = & H^+_{Y,K}, \\
\mathcal{S}_{G^+_{X,K}} \cup \left\lbrace p^+_{Z_k} ~|~~k \in I_\leftarrow \right\rbrace & = & \mathcal{S}_{G^+_{Y,K}} \cup \left\lbrace p^+_{Z_k} ~|~~k \in I_\rightarrow \right\rbrace.
\end{eqnarray*}
Since $G^+_{X,K}$ and $G^+_{Y,K}$ are integral, then we write the isometry
$$U^+_X \oplus \bigoplus_{k=1}^{n_X} L^+_{X,k} \oplus \bigoplus_{k \in I_\leftarrow} T^+_{Z_k}(-1) \simeq U^+_Y \oplus \bigoplus_{k=1}^{n_Y} L^+_{Y,k} \oplus \bigoplus_{k \in I_\rightarrow} T^+_{Z_k}(-1),$$
and this decomposition is unique.
Therefore, we obtain $U^+_X \simeq U^+_Y$ and the sets
$$\left\lbrace L^+_{X,k} ~|~~1 \leq k \leq n_X \right\rbrace \cup \left\lbrace T^+_{Z_k}(-1) ~|~~k \in I_\leftarrow\right\rbrace$$
and
$$\left\lbrace L^+_{Y,k} ~|~~1 \leq k \leq n_Y \right\rbrace \cup \left\lbrace T^+_{Z_k}(-1) ~|~~k \in I_\rightarrow\right\rbrace$$
have isometric elements.
In particular, we deduce the following.

\begin{pro}\label{rat}
Let $X$ be a rational fourfold.
Then $G^+_{X,\QQ}=\QQ^{\oplus n_X}$ for some integer $n_X$, and in every weak factorization from $X$ to $\PP^4$, there exist $n_X$ blow-up center surfaces $Z_1, \dotsc, Z_{n_X}$ such that
$$T^+_X \simeq T^+_{Z_1}(-1) \oplus \dotsb \oplus T^+_{Z_{n_X}}(-1)$$
as an isometry of $\ZZ$-Hodge structures.
\end{pro}

\subsection{Relation to Hodge atoms}
Recall that a Hodge atom is a generalized eigenspace of the endomorphism $\ev(\kappa_\tau)$ for a general evaluation map, see \cite{JG}.
By Lemmas \ref{lem1} and \ref{lem2} and by Proposition \ref{initial}, the atom construction is compatible with the discussion right before Proposition \ref{rat}.
This means in particular that $T^+_X$ and each $T^+_{Z_k}(-1)$ for $k \in I_{\leftarrow}$ all belong to distinct Hodge atoms, and so do $T^+_Y$ and each $T^+_{Z_k}(-1)$ for $k \in I_{\rightarrow}$.
\begin{pro}\label{rat2}
Let $X$ be a rational fourfold.
Then for each Hodge atom $E_{\ev,\alpha}$ of $X$, in every weak factorization from $X$ to $\PP^4$, there exists a unique blow-up center surface $Z$ whose minimal model $\Sigma$ satisfies
$$T^+_X \cap E_{\ev,\alpha} \simeq T^+_\Sigma(-1)$$
as an isometry of $\ZZ$-Hodge structures.
\end{pro}

We can then refine \cite[Theorem 63]{JG}, \cite[Theorem 13]{GM}, and \cite[Theorem 1]{Verra}.
\begin{thm}\label{cubic}
Let $X$ be a rational cubic, Gushel--Muka\"i, or Verra fourfold.
Then there exists a projective K3 surface $\Sigma$ and an isometry of integral Hodge structures
$$H^4(X,\ZZ)_\mathrm{van} \simeq H^2(\Sigma,\ZZ)(-1).$$
\end{thm}

\begin{proof}
By the cited theorems, the result holds with $\QQ$-coefficients, and by Proposition \ref{rat2}, it holds for the transcendantal lattice.
\end{proof}

As a consequence, we find the following result, refining \cite[Theorem 6.4]{Fay}.

\begin{cor}\label{Verra}
A very general Verra fourfold is irrational.
\end{cor}

\begin{proof}
A Verra fourfold $X$ is a double cover $\pi \colon X \to \PP^2 \times \PP^2$ ramified at a divisor of bidegree $(2,2)$.
Let $A \subset H^4(X,\ZZ)$ be the sublattice obtained from the map $\pi$:
$$A = \ZZ \langle h_1^2, h_1 h_2, h_2^2 \rangle,$$
where $h_1$ and $h_2$ are the hyperplane classes of each $\PP^2$ pulled-back via $\pi$.
Its intersection matrix is
$$\begin{bmatrix}
0 & 0 & 2 \\
0 & 2 & 0 \\
2 & 0 & 0
\end{bmatrix},$$
so that its discriminant group is $\(\ZZ/2\ZZ\)^3$.
Furthermore, $A \subset H^4(X,\ZZ)$ is saturated, i.e.~the quotient has no torsion (this is true for every smooth Verra fourfold). 
Our assumption that $X$ is very general actually means that
$$T_X = A^\perp.$$
Since $H^4(X,\ZZ)$ is unimodular and $A$ is saturated, then $T_X$ has the same discriminant group $\(\ZZ/2\ZZ\)^3$ as $A$.

Assume that there is a K3 surface $\Sigma$ and an isometry
$$T_X \simeq T_\Sigma(-1).$$
Then the Picard number of $\Sigma$ is $1$ and we have $T_\Sigma = \mathrm{NS}(\Sigma)^\perp$ with $\mathrm{NS}(\Sigma) = \ZZ \langle h \rangle$ generated by an ample class.
The lattice $H^2(\Sigma,\ZZ)$ has no torsion and is unimodular, and the sub-lattice $\mathrm{NS}(\Sigma)$ is saturated.
Hence, the discriminant of $T_\Sigma$ equals the discriminant of $\mathrm{NS}(\Sigma)$, which is the cyclic group $\ZZ/2d\ZZ$ where $h^2=:2d$.
Since $\(\ZZ/2\ZZ\)^3$ is not cyclic, then we get a contradiction.
\end{proof}


\begin{bibdiv}
	\begin{biblist}

		\bib{Verra}{article}{
	author={Benedetti, Vladimiro},
	author={Gu\'er\'e, J\'er\'emy},
	author={Manivel, Laurent},
	author={Perrin, Nicolas},
	title={Quantum cohomology and birational geometry of Verra fourfolds},
	journal={arXiv 2605.30450},
	volume={},
	date={},
	number={},
	pages={},
}	

		\bib{GM}{article}{
	author={Benedetti, Vladimiro},
	author={Manivel, Laurent},
	author={Perrin, Nicolas},
	title={uantum cohomology and irrationality of Gushel-Mukai fourfolds},
	journal={arXiv:2603.17487},
	volume={},
	date={},
	number={},
	pages={},
}		
		
	
		\bib{Fay}{article}{
	author={Fay, Aideen},
	title={Equivariant irrationality of very general symmetric Verra fourfolds},
	journal={arXiv 2605.30439},
	volume={},
	date={},
	number={},
	pages={},
}

		\bib{Fay2}{article}{
	author={Fay, Aideen},
	title={Hasse obstructions to rationality for special fourfolds},
	journal={arXiv 2609.06759},
	volume={},
	date={},
	number={},
	pages={},
}

		
		\bib{JG}{article}{
	author={Gu\'er\'e, J\'er\'emy},
	title={On the irrationality of cubic fourfolds},
	journal={arXiv 2603.04518},
	volume={},
	date={},
	number={},
	pages={},
}

		\bib{Hassett}{article}{
			author={Hassett, Brendan},
			title={Special cubic fourfolds},
			journal={Compositio Mathematica},
			volume={120},
			date={2000},
			number={1},
			pages={1-23},
		}
		
		\bib{Iritani}{article}{
			author={Iritani, Hiroshi},
			title={Quantum cohomology of blow-ups},
			journal={to appear in Kyoto Journal of Mathematics, avalaible at arXiv 2307.13555},
			volume={},
			date={},
			number={},
			pages={},
		}
		

		
		\bib{KKPY}{article}{
			author={Katzarkov, Ludmil},
			author={Kontsevich, Maxim},
			author={Pantev, Tony},
			author={Yu Yue, Tony},
			title={Birational invariants from Hodge structures and quantum multiplication},
			journal={avalaible at arXiv 2508.05105},
			volume={},
			date={},
			number={},
			pages={},
		} 
		
		
		
		\bib{Kuznetsov}{article}{
			author={Kuznetsov, Alexander},
			title={Derived categories view on rationality problems},
			journal={Lecture Notes in Math.},
			volume={2172},
			date={2016},
			number={},
			pages={67-104},
		}

		\bib{Lin}{article}{
	author={Lin, Hsueh-Yung},
	author={Shinder, Evgeny},
	title={Motivic invariants of birational maps},
	journal={Annals of Mathematics},
	volume={199},
	date={2024},
	number={1},
	pages={445-478},
}
		
		\bib{OpenAI}{article}{
			author={OpenAI},
			title={Irrational cubic fourfolds with geometric K3 categories},
			journal={},
			volume={},
			date={},
			number={},
			pages={},
		}
		
	\end{biblist}
\end{bibdiv}
\end{document}